\documentclass[12pt]{amsart}
\usepackage{amssymb}
\usepackage{mathrsfs}
\usepackage[headings]{fullpage}

\numberwithin{equation}{section}

\theoremstyle{plain}
\newtheorem{thm}[subsection]{Theorem}

\newtheorem{lem}[subsection]{Lemma}

\theoremstyle{definition}

\theoremstyle{remark}
\newtheorem{rem}[subsection]{Remark}

\def\Z {{\mathbb Z}}

\def\Q {{\mathbb Q}}

\def\F {{\mathbb F}}

\def\CC {{\mathcal C}}
\def\EE {{\mathcal E}}

\def\O {{\mathcal O}}

\newcommand{\GL}{\textnormal{GL}}
\newcommand{\SL}{\textnormal{SL}}

\newcommand{\Aut}{\textnormal{Aut}}
\newcommand{\Char}{\textnormal{Char}}

\newcommand{\Frob}{{\textnormal{Frob}}}
\newcommand{\Gal}{\textnormal{Gal}}

\newcommand{\ord}{{\textnormal{ord}}}

\newcommand{\Spec}{{\textnormal{Spec}}}

\begin{document}
\title[Most Elliptic Curves Over Global Function Fields are Torsion Free]{Most Elliptic Curves Over Global Function Fields are Torsion Free}

\author{Tristan Phillips}
\address{Department of Mathematics \\ University of Arizona
  \\ Tucson, AZ~~85721 USA}
\email{tristanphillips@math.arizona.edu}



\begin{abstract}
Given an elliptic curve $E$ over a global function field $K$, the Galois action on the $n$-torsion points of $E$ gives rise to a mod-n Galois representation $\rho_{E,n}$. For $K$ satisfying some mild conditions, we show that the set of $E$ for which $\rho_{E,n}$ is as large as possible for all $n$, has density $1$.
\end{abstract}

\maketitle

\section{Introduction}

Let $E$ be an elliptic curve defined over a global field $K$. For each positive integer $n$ relatively prime to the characteristic of $K$  there is an action of the absolute Galois group $G_K:=\Gal(K^{sep}/K)$ on the $n$-torsion points $E[n](K^{sep})$ of $E$ which induces a representation
\[
\rho_{E,n}:G_K\to \Aut(E[n](K^{sep}))\cong \GL_{2}(\Z/n\Z).
\]

It is natural to ask how large the image of $\rho_{E,n}$ is. A fundamental result of Serre says that if $E$ is a non-CM elliptic curve over $\Q$, then $\rho_{E,p}$ is surjective for all but finitely many primes $p$ \cite{Ser72}. Duke has shown that the set of elliptic curves over $\Q$ for which $\rho_{E,n}$ is surjective for all $n$ has density 1, when counted by naive height \cite{Duk97}. Call an elliptic curve \emph{torsion free} if its Mordell-Weil group $E(\Q)$ is torsion free. Since $\rho_{E,n}$ being surjective implies $E[n](\Q)=\emptyset$, a pleasing consequence of Duke's result is that the set of torsion free elliptic curves over $\Q$ has density 1.

In order to get analogous results for more general $K$, one must restrict the codomain of $\rho_{E,n}$. Let $\chi_n$ be the cyclotomic character; it is well known that $\chi_n=\det\circ \rho_{E,n}$. Denote by $\Gamma_n\subseteq \GL_2(\Z/n\Z)$ the subgroup determined by the exact sequence 
\[
0\rightarrow\SL_2(\Z/n\Z)\rightarrow \Gamma_n\xrightarrow{\det} \chi_n(G_K) \rightarrow 0.
\]
Zywina has shown that when $K$ is a number field, the set of elliptic curves for which the image of $\rho_{E,n}$ equals $\Gamma_n$ for all $n$ has density 1 \cite{Zyw10}. This has the consequence that, over any number field, the set of torsion free elliptic curves has density 1.

When $K$ is a global function field, a classical result of Igusa says that for any non-isotrivial elliptic curve $E$ and any $n$ relatively prime to $\Char(K)$, the image of $\rho_{E,n}$ equals $\Gamma_n$ for all but finitely many $n$ \cite{Igu59} (see also \cite{BLV09}). It is then natural to ask for analogs of the results of Duke and Zywina. In this article we shall prove such an analog.

\section{Statement of Main Result}

Let $\F_q$ be a finite field of characteristic $p>3$. Let $K$ be the function field of a projective smooth genus $g$ curve $\mathscr{C}$ over $\F_q$. Assume that there exists a degree $1$ rational point  $\infty$ on $\mathscr{C}$ (this assumption is needed in order to apply a version of the large sieve inequality). Let $\O_K$ be the set of functions on $\mathscr{C}$ regular away from $\infty$, let $\ord_\infty:K^\times\to \Z$ be the valuation at $\infty$, and let $K_\infty$ denote the completion of $K$ with respect to the absolute value $|\cdot|_\infty:=q^{-ord_\infty f}$. 

Let $E$ be an elliptic curve defined over $K$. As $\Char(K)>3$, $E$ has a model
\[
E_{(a,b)}:y^2=x^3+ax+b
\]
with $a,b\in \O_K$. Conversely, if $\Delta_{(a,b)}:=-16(4a^3+27b^2)$ is non-zero then $E_{(a,b)}$ defines an elliptic curve over $K$. 

An elliptic curve $E$ over a field of characteristic $p$ is \emph{ordinary} if $E[p](K^{sep})\cong \Z/p\Z$, otherwise $E[p](K^{sep})$ is trivial and we call $E$ \emph{supersingular}. 

For a prime $\ell\neq p$ we define $\rho_{E,\ell}$ as in the introduction.  If $E$ is ordinary then the action of $G_K$ on $E[p]$ induces a representation
\[
\rho_{E,p}:G_K\to \Aut(E[p])\cong (\Z/p\Z)^\times.
\]
If $E$ is supersingular set $\rho_{E,p}(G_K)=\{1\}$.

Define the sets
\begin{align*}
\CC(x)&:=\{(a,b)\in \O_K^2 : \Delta_{(a,b)}\neq 0,\  \max(|a|_\infty,|b|_\infty)\leq q^x\} \\
\EE_\ell(x)&:=\{(a,b)\in \CC(x) : \rho_{E_{(a,b)},\ell}(G_K)\neq\Gamma_\ell\}\\
\EE_p(x)&:=\{(a,b)\in \CC(x) : \rho_{E_{(a,b)},p}(G_K)\neq(\Z/p\Z)^\times \}\\
\EE(x)&:=\bigcup_{\ell} \EE_\ell(x).
\end{align*}
where the union is over all primes, including $\ell=p$.

Let $f,g$ be positive real valued functions. Then $f\sim g$ means that $\lim_{x\to\infty} f(x)/g(x)=1$, and $f\ll g$ means that there exists a constant $c$ such that for all sufficiently large $x$, $f(x)\leq c\cdot g(x)$. We sometimes use the alternate notation $f=O(g)$ for $f\ll g$. Subscripts on $\ll$ and $O$ will be used to denote the variables the implied constant depends on.

Applying the Riemann-Roch theorem to the divisor $x\infty$ on the curve $\mathscr{C}$ we see that
\begin{align}\label{C}
\# \CC(x)\sim c q^{2x}
\end{align}
where $c$ is a positive constant depending on $K$.

The following result of Cojocaru and Hall gives an explicit version of the previously mentioned result of Igusa: 

\begin{thm}{{\cite[Theorem 1.1]{CH05}}}\label{thm:uniformity}
Define a constant, depending only on $g=\textnormal{genus}(K)$, by
\[
c(g):=2+\max\left\lbrace\ell: \frac{\ell-(6+3e_4+4e_3)}{12}\leq g\right\rbrace ,
\]
where 
\[
e_j:=\begin{cases}
1 & \text{ if } \ell\equiv 1\pmod{j}\\
-1 & \text{ else.}
\end{cases}
\]
Then $\rho_{E,\ell}(G_K)$ equals $\Gamma_\ell$ for all $\ell\geq c(g)$.
\end{thm}

\noindent We now state our main result:

\begin{thm}\label{thm:Main}
Let $K$ be a genus $g$ global function field of characteristic $p$. If there are no primes $\ell\neq p$ less than $c(g)$ such that $p|(\ell\pm 1)$, then
\[
\frac{\#\EE(x)}{\#\CC(x)}\ll_{K} \frac{x}{q^{x/2}}.
\]
In particular,
\[
\lim_{x\to\infty} \frac{\#\EE(x)}{\#\CC(x)}=0.
\]
\end{thm}

\begin{rem}
The condition $p\nmid (\ell\pm 1)$ in the theorem is needed in order to apply a version of the Chebotarev density theorem for varieties over finite fields (Lemma \ref{lem:FiniteCheb}). 
\end{rem}

The theorem implies that the set of torsion free elliptic curves over $K$ has density 1. Our proof uses a multidimensional large sieve for global function fields and is similar to the argument in the number field case.

\section{Large Sieve}
 Define the \emph{degree} of an ideal $I\subset \O_K$ as $\deg(I):=\dim_{\F_q}(\O_K/I\O_K)$. 

Extend $|\cdot|_\infty$ to $K_\infty^2$ by
\[
|(f_1,f_2)|_\infty=\max\{|f_1|_\infty, |f_2|_\infty\}.
\]

The following is a 2-dimensional large sieve inequality for global function fields:

\begin{thm}[Large Sieve]
Let $Q,R\in \Z_{\geq 0}$. For each prime ideal $P$ of $\O_K$ let $\omega_P$ be a real number in $[0,1)$. Let $W$ be a subset of $\O_K^2$ such that 
\[
\# W_P\leq(1-\omega_P) \cdot q^{2 \deg(P)}
\]
where $W_P$ denotes the canonical image of $W$ in $(\O_K/P O_K)^2$. Then
\[
\# \{w\in W: |w|_\infty\leq q^R\} \leq \frac{q^{2\cdot \max\{R+1,2Q+2g\}}}{L(Q)},
\]
where
\[
L(Q)=1+\sum_{I\in S_Q} \prod_{P|M} \frac{\omega_P}{1-\omega_P}
\]
where $S_Q$ denotes the set of squarefree ideals $I\subseteq \O_K$ with $\deg(I)\leq Q$.
\end{thm}

\begin{proof}
The proof of the $K=\F_q(T)$ case given in \cite[Theorem 3.2]{Hsu96} carries over verbatim for our more general $K$ (see \cite[Theorem 3.2]{Hsu99} for the 1-dimensional case). 
\end{proof}

\section{Elliptic Curves over Finite Fields}

Let $E$ be an elliptic curve over $\F_{q^n}$. For a rational prime $\ell\neq p$ we get a representation
\[
\overline{\rho}_{E,\ell}:\Gal(\overline{\F}_{q^n}/\F_{q^n})\to\Gamma_\ell\subset\GL_2(\Z/\ell\Z),
\]
and for $\ell=p$,
\[
\overline{\rho}_{E,p}:\Gal(\overline{\F}_{q^n}/\F_{q^n})\to(\Z/p\Z)^\times
\]
where $\overline{\rho}_{E,p}(\overline{\F}_{q^n}/\F_{q^n})=\{1\}$ if $E$ is supersingular.

For $(a,b)\in \F^2_{q^n}$ with $\Delta_{(a,b)}:=-16(4a^3+27b^2)$ non-zero, let $E_{(a,b)}$ be the elliptic curve 
\[
E_{(a,b)}:y^2=x^3+ax+b.
\]

For $\ell\neq p$ and $C$ a conjugacy class of $\Gamma_\ell$ set 
\[
\Omega_{\ell,C}(n):=\{(a,b)\in \F_{q^n}^2 : \Delta_{(a,b)}\neq 0,\ \overline{\rho}_{E_{(a,b)},\ell}(\Frob_{q^n})\in C\}.
\]

For $\ell=p$ and $t\in(\Z/p\Z)^\times$ set
\[
\Omega_{p,t}(n):=\{(a,b)\in \F_{q^n}^2 : \Delta_{(a,b)}\neq 0,\ \overline{\rho}_{E_{(a,b)},p}(\Frob_{q^n})=t\}
\]


 Applying the Chebotarev density theorem for finite fields, as given in \cite[Theorem 1]{Kow06}, with $U=\Spec(\F_q[a,b,1/(4a^3+27b^2)])$ gives:

\begin{lem}\label{lem:FiniteCheb}
For any $\ell$ such that $p\nmid \#\SL_2(\Z/\ell\Z)=\ell(\ell+1)(\ell-1)$, let $C\subset \Gamma_\ell$ be a subset closed under conjugation. Then, for any $n$ such that $q^n\equiv \det(C)\pmod{\ell}$,
\[
\#\Omega_{\ell,C}(n)=\frac{\# C}{\#\SL_2(\Z/\ell\Z)}q^{2n}+O\left(q^{3n/2}\cdot\sqrt{\# C\cdot \# \GL_2(\Z/\ell\Z)^3}\right)
\]
where the implied constant is absolute.
For $\ell=p$ and $t\in (\Z/p\Z)^\times$, 
\[
\#\Omega_{p,t}(n)=\frac{1}{p-1}q^{2n}+O\left(q^{3n/2}(p-1)^{3/2}\right)
\]
where the implied constant is absolute.
\end{lem}

\begin{rem}
The arguments in the proof of Theorem 8 in \cite{Jon10} can be generalized to give a proof of Lemma \ref{lem:FiniteCheb} different from Kowalski's proof. 
\end{rem}

\section{Estimating $\#\EE_\ell(x)$}

For $\ell\neq p$ and $d$ relatively prime to $\ell$, let $\Sigma_K(Q;\ell,d)$ denote the set of prime ideals $P\subset \O_K$ with $\deg(P)\leq Q$ and 
\[
q^{\deg(P)}\equiv d\pmod{ \ell}.
\]
We may suppose $q^Q\equiv d\pmod{\ell}$; if not we can decrease $Q$ so that it satisfies this condition. Under this assumption, the prime polynomial theorem implies
\begin{align}\label{ppnt}
\frac{q^{Q}}{Q} \ll \#\Sigma_K(Q;\ell,d) 
\end{align}
where the implied constant is absolute.

\begin{lem}\label{lem:modl}
For any prime $\ell\neq p$ satisfying $p\nmid (\ell\pm 1)$,
\[
\frac{\# \EE_\ell(x)}{\# \CC(x)} \ll_{K,\ell} \frac{x}{q^{x/2}}.
\]
\end{lem}

\begin{proof}

Let $C$ be a conjugacy class of $\GL_2(\Z/\ell\Z)$ with $d=\det(C)$. Set
\[
W_C(x):=\{(a,b)\in \CC(x) : \rho_{E_{(a,b)},\ell}(G_K\cap C)=\emptyset\}.
\]
Let $C_1,\dots,C_m$ be the determinant 1 conjugacy classes of $\GL_2(\Z/\ell\Z)$. 
By \cite[Lemma A.10]{Zyw10} $\EE_\ell(x)=\bigcup_{i=1}^m W_{C_i}(x)$. Hence
\[
\frac{\# \EE_\ell(x)}{\# \CC(x)}\leq \sum_{i=1}^m \frac{\# W_{C_i}(x)}{\# \CC(x)}.
\]

We now use the large sieve to estimate $\# W_C(x)$. Take $R=x$ and $Q=x/2$. For $P\in \Sigma_K(Q;\ell,d)$ set $\F_P=\O_K/P\O_K$, $\Omega_P:=\Omega_C(\deg(P))$, and
\[
\omega_P:=\frac{\#\Omega_P}{q^{2\deg(P)}}.
\]
Denote by $W_P$ the image of $W_C(X)$ in $\F_P^2$. 

Let $\Frob_P$ denote the $q^{\deg(P)}$ Frobenius endomorphism. If $(a,b)\in\EE_\ell(x)$ is such that $(a,b)\mod P$ is in $\Omega_P$ then $\overline{\rho}_{E_{(a,b)},\ell}(\Frob_P)\in C$ implies $(a,b)\notin W_C(x)$. Hence $W_P\subset \F_P^2\setminus \Omega_P$, which shows $\# W_P\leq (1-\omega_P)\cdot q^{2\deg(P)}$. Therefore, by the large sieve inequality,
\[
\# W_C(x) \leq \frac{q^{2\cdot \max\{x+1,x+2g\}}}{L(Q)}\ll_{q,g} \frac{q^{2x}}{L(Q)}
\]
where 
\[
L(Q)= 1+\sum_{I\in S_Q} \prod_{P|M} \frac{\omega_P}{1-\omega_P} \geq \sum_{P\in \Sigma_K(Q;N,A)} \omega_P.
\]
But by Lemma \ref{lem:FiniteCheb},
\[
\omega_P=\frac{\# \Omega_C(\deg(P))}{q^{2 \deg(P)}}=\frac{\# C}{\#\SL_2(\Z/\ell\Z)}+O\left(\frac{\sqrt{\#C\cdot \#\GL_2(\Z/\ell\Z)^3}}{q^{\deg(P)/2}}\right).
\]

Therefore
\begin{align*}
L(Q)&\geq \sum_{P\in \Sigma_K(Q;N,A)}\left(\frac{\# C}{\#\SL_2(\Z/\ell\Z)}+O\left(\frac{\sqrt{\#C\cdot \#\GL_2(\Z/\ell\Z)^3}}{q^{\deg(P)/2}}\right)\right)\\
&\sim \#\Sigma_K(Q;N,A) \frac{\# C}{\#\SL_2(\Z/\ell\Z)}.
\end{align*}

Hence
\[
\# W_C(x)\ll_{q,g,\ell} \frac{q^{2x}}{ \# \Sigma_K(Q;N,A)}.
\]

From this, together with (\ref{ppnt}) and (\ref{C}), we find that
\begin{align*}
\frac{\# \EE_\ell(x)}{\# \CC(x)} &\leq \sum_{i=1}^m \frac{\# W_{C_i}(x)}{\# \CC(x)}
 \ll_{K,\ell} \sum_{i=1}^m \frac{1}{ \# \Sigma_K(Q;N,A)}
\ll \frac{x}{q^{x/2}}.
\end{align*}

\end{proof}

\section{Estimating $\#\EE_p(x)$}

Let $\Sigma_K(Q)$ denote the set of prime ideals $P\subset \O_K$ with $\deg(P)\leq Q$. By the prime polynomial theorem
\begin{align}\label{ppt}
\frac{q^{Q}}{Q} \ll \#\Sigma_K(Q)
\end{align}
where the implied constant is absolute.

\begin{lem}\label{lem:modp}
\[
\frac{\# \EE_p(x)}{\# \CC(x)} \ll_K \frac{x}{q^{x/2}}.
\]
\end{lem}

\begin{rem}
This result says ``$\rho_{E,p}$ is usually surjective." It is interesting to compare this with Igusa's theorem, which implies that the mod $p$ Galois representation of the \emph{universal elliptic curve} over $K$ is surjective \cite{Igu68}. 
\end{rem}

\begin{proof}

We again use the large sieve. Let $t\in (\Z/p\Z)^\times$ be a generator. The setup is similar to the $\ell\neq p$ case:
\begin{align*}
W_t(x)&:=\{(a,b)\in \CC(x) : t\not\in\rho_{E_{(a,b)},p}(G_K)\}.\\
R&=x\\
Q&=x/2.
\end{align*}
For each prime $P\subset \O_K$,
\begin{align*}
\Omega_P&:=\Omega_{p,t}(\deg(P))\\
\omega_P&:=\frac{\#\Omega_P}{q^{2\deg(P)}}\\
W_P&:=\text{image of $W_t(X)$ in $(\O_K/P \O_K)^2$}
\end{align*}

Note that $W_t(x)=\EE_p(x)$.
 
If $(a,b)\in\CC(x)$ is such that $(a,b)\mod P$ is in $\Omega_P$ then $\overline{\rho}_{E_{(a,b)},p}(\Frob_P)=t$ implies $(a,b)\notin W_t(x)$. Hence $W_P\subset \F_P^2\setminus \Omega_P$, which shows $\# W_P\leq (1-\omega_P)\cdot q^{2\deg(P)}$. Therefore, by the large sieve inequality,
\[
\# W_t(x) \leq \frac{q^{2\cdot\max\{x+1, x+2g\}}}{L(Q)}\ll_{q,g} \frac{q^{2x}}{L(Q)},
\]
where 
\[
L(Q)= 1+\sum_{I\in S_Q} \prod_{P|M} \frac{\omega_P}{1-\omega_P} \geq \sum_{P\in \Sigma_K(Q)} \omega_P.
\]
By Lemma \ref{lem:FiniteCheb},
\[
\omega_p=\frac{1}{p-1}+O\left(\frac{(p-1)^{3/2}}{q^{\deg(P)/2}}\right).
\]

Therefore
\begin{align*}
L(Q)
\geq \sum_{P\in \Sigma_K(Q)}\left(\frac{1}{p-1}+O\left(\frac{(p-1)^{3/2}}{q^{\deg(P)/2}}\right)\right)
\sim \frac{\# \Sigma_K(Q)}{p-1}.
\end{align*}

Hence
\[
\# W_t(x)\ll_{q} \frac{q^{2x}}{ \# \Sigma_K(Q)}.
\]

From this, together with (\ref{ppt}) and (\ref{C}), it follows that 
\begin{align*}
\frac{\# \EE_\ell(x)}{\# \CC(x)} = \frac{\# W_{t}(x)}{\# \CC(x)}
\ll_{K} \frac{1}{\#\Sigma_K(Q)}
\ll_K \frac{x}{q^{x/2}}.
\end{align*}

\end{proof}

\section{Estimating $\#\EE(x)$}

\noindent We now prove the main theorem:

\begin{proof}

By Theorem \ref{thm:uniformity}
\[
\EE(x)=\bigcup_{2\leq \ell < c(K)} \EE_\ell(x).
\]
By Lemma \ref{lem:modl} and Lemma \ref{lem:modp}
\[
\sum_{2\leq \ell< c(g)} \frac{\#\EE_\ell(x)}{\# \CC(x)} \ll_{K} \frac{x}{q^{x/2}}.
\]
\end{proof}

\subsection*{Acknowledgments}
The author thanks Doug Ulmer and Bryden Cais for their guidance and encouragement during this project, and also Nathan Jones for clarifying a remark in \cite{Jon10}.

\bibliographystyle{alpha}
\bibliography{bibfile}

\end{document}